\documentclass[11pt, reqno]{amsart}   	
		
\usepackage{amssymb, latexsym}
\usepackage{amsmath}
\usepackage[toc,page]{appendix}
\usepackage{braket}
\usepackage{breqn}
\usepackage{caption}
\usepackage{color}
\usepackage{esint}
\usepackage[T1]{fontenc}
\usepackage{geometry}         
\usepackage{graphicx}
\usepackage{hyperref}
\usepackage{latexsym}
\usepackage{mathrsfs}
\usepackage{mathtools}
\usepackage{times}
\usepackage{url} 

\theoremstyle{plain}
\numberwithin{equation}{section}
\newtheorem{thm}{Theorem}[section]
\newtheorem{theorem}[thm]{Theorem}
\newtheorem{lemma}[thm]{Lemma}

\newtheorem{remark}[thm]{Remark}

\def\R{\mathbb{R}}

\def\N{\mathbb{N}}

\def\fa{\forall}

\def\Lin{\mathcal{L}}
\def\Meas{\mathcal{M}}

\author[Joshua M. Siktar]{Joshua M. Siktar$^1$}
\address{$^1$Department of Mathematical Sciences,
\newline \indent Carnegie Mellon University,
\newline \indent Pittsburgh, PA 15213, United States}
\email{jsiktar@alumni.cmu.edu}

\keywords{Parseval's Identity, Fourier Series, Bessel's Inequality}
\title{Recasting the Proof of Parseval's Identity}
\date{\today}							

\begin{document}

\begin{abstract}
We generalize aspects of Fourier Analysis from intervals on $\R$ to bounded and measurable subsets of $\R^n$. In doing so, we obtain a few interesting results. The first is a new proof of the famous Integral Cauchy-Schwarz Inequality. The second is a restatement of Parseval's Identity that doubles as a representation of integrating bounded and measurable functions over bounded and measurable subsets of $\R^n$. Finally, we apply these first two results to develop some sufficient criteria for additional integral inequalities that are elementary in nature.
\end{abstract}

\maketitle

\tableofcontents

\section{Introduction and Motivation} \label{intro}

In a typical first study of partial differential equations, great attention is devoted to Fourier Analysis, namely to the derivation of formulas for Fourier Coefficients and analyzing when a function has a Fourier Expansion over some subset of the real line. This is evident upon inspecting textbooks such as \cite{Wei, Za}. Further analysis of Fourier Coefficients over intervals in $\R$ is discussed in \cite{Ap, Liu, Sik}\footnote{Also see the unpublished article "Fun with Fourier Series" \url{https://arxiv.org/pdf/0806.0150.pdf} by R. Baillie}\footnote{Also see the unpublished article "Approximations for Apery's Constant $\zeta(3)$ and Rational Series Representations Involving $\zeta(2n)$," \url{https://arxiv.org/pdf/1605.09541.pdf} by C. Lupu and D. Orr}, and from these it is evident that Fourier Coefficients have use in mathematics beyond analysis, such as in number theory. 

However, in introductory texts the emphasis is usually on Fourier Coefficients of a function over an interval. If we use the fundamental tools of measure theory, from books such as \cite{Rie, Roy}, then we not only obtain results from Fourier Analysis in more generality, but also find a new proof of the Integral Cauchy-Schwarz Inequality. From here on, we will assume for convenience that we are over the standard Lebesgue Measure space in $\R^n$, denoted $(\R^n, \Meas, \Lin^n)$. 

Section \ref{integralCSProof} will be devoted to demonstrating the aforementioned new proof. Then, Section \ref{parsevalGen} will prove a special case of Parseval's Identity (which is discussed in \cite{Sik}) using the proof from \ref{integralCSProof} and tools from \cite{Roy}. In particular, the following lemma will be used repeatedly in Section \ref{parsevalGen}:

\vspace{0.2 cm}

\begin{lemma}[Countable Additivity of Integration]
Let $f$ be a measurable function defined on the measurable set $E \subset \R^n$. Let $\{E_n\}^{\infty}_{n = 1}$ be a disjoint, countable collection of measurable subsets of $E$ whose union is $E$. Then

\begin{equation}
\int_E f d\mu \ = \ \sum^{\infty}_{n = 1}\int_{E_n}f d\mu.
\label{countableIntegralSum}
\end{equation}

\end{lemma}

Finally, Section \ref{productIneq} serves as an application of the results in Sections \ref{integralCSProof} and \ref{parsevalGen}; namely, we will describe sufficient conditions for when one can compare the integral of the product of two functions to the product of the integrals of the two functions. 

\vspace{0.2 cm}

\section{Proof of Integral Cauchy-Schwarz Inequality}
\label{integralCSProof}

The main theorem to be proven in this section is the following:

\vspace{0.2 cm}

\begin{theorem}[Integral Cauchy-Schwarz]
Let $E \subset \R^n$ be a bounded and measurable set, and let $g, h: E \rightarrow \R$ be bounded and measurable functions. Then

\begin{equation}
\left(\int_E g^2d\mu\right)\left(\int_E h^2 d\mu\right) \ \geq \ \left(\int_E gh d\mu\right)^2.
\label{integralCS}
\end{equation}

\end{theorem}

We will break the proof of \eqref{integralCS} into lemmas that act as a reduction of \eqref{integralCS} onto the case where the functions over which we integrate are strictly positive over the domain $E$.

\vspace{0.2 cm}

\begin{lemma}
Let $D \subset \R^n$ be a bounded and measurable set, and let $f, \phi_1: D \rightarrow \R$ be bounded and measurable functions, where $f$ only takes positive values in $D$. Then

\begin{equation}
\left(\int_D f d\mu\right)\left(\int_D \frac{\phi_1^2}{f}d\mu\right) \ \geq \ \left(\int_D \phi_1d\mu\right)^2.
\label{integralCSLemma1}
\end{equation}
\end{lemma}

\begin{proof} Fix the function $\phi_1$ and construct a family of functions $\phi_2, \phi_3, ...$ such that the collection $\{\phi_i\}^{\infty}_{i = 1}$ is mutually orthogonal on $D$ with respect to the [positive] weight function $\frac{1}{f}$. That is, $\fa i \neq j$,

\begin{equation}
\int_D \phi_i\phi_j \cdot \frac{1}{f}d\mu \ = \ 0.
\label{orthogonality}
\end{equation}

If there does not exist an infinite family of mutually orthogonal functions that includes $\phi_1$, we truncate the family after including some $k \geq 1$ functions (only the value $\phi_1$ will have any relevance at the end of the proof). Now we define the sequence of partial sums $s_N := \sum^{N}_{n = 1}c_n\phi_n$, where for each $n$ we set

\begin{equation}
c_n \ := \ \frac{\int_D \phi_n d\mu}{\int_D \frac{\phi_n^2}{f}d\mu}.
\label{fixFourierCoefficients}
\end{equation}

If our family of mutually orthogonal functions $\{\phi_i\}$ only contains $k$ functions, then we set $s_n :=  s_k$ whenever $n > k$; in this case, the sequence of partial sums is said to be \textbf{eventually constant}. Now the inequality \eqref{integralCSLemma1} will follow from an attempt to minimize the \textbf{mean-square deviation integral}

\begin{equation}
\int_D (f - s_N)^2 \cdot \frac{1}{f}d\mu,
\label{meanSquareDeviationIntegral}
\end{equation}

much akin to how \cite{Wei} derives Parseval's Identity, except here we are considering the special case where $\rho := \frac{1}{f}$. Since $f > 0$ on $D$, the integral \eqref{meanSquareDeviationIntegral} is nonnegative for all $N \in \N^+$. We can expand the integral \eqref{meanSquareDeviationIntegral} and complete the square; due to the mutual orthogonality of the functions $\{\phi_n\}^{N}_{n = 1}$, \eqref{meanSquareDeviationIntegral} in fact equals

\begin{eqnarray}
\int_D f d\mu - 2\sum^{N}_{n = 1}c_n\int_D \phi_n d\mu + \sum^{N}_{n = 1}c_n^2\int_D \frac{\phi_n^2}{f}d\mu \ = \ \nonumber\\
 \sum^{N}_{n = 1}\int_D \frac{\phi_n^2}{f}d\mu\left(c_n - \frac{\int_D \phi_n d\mu}{\int_D \frac{\phi_n^2}{f}d\mu}\right)^2 + \int_D f d\mu - \sum^{N}_{n = 1}\frac{\left(\int_D \phi_n d\mu\right)^2}{\int_D \frac{\phi_n^2}{f}d\mu}. \label{completeFourierSquare}
\end{eqnarray}

Due to our choice of coefficients $c_n$ in \eqref{fixFourierCoefficients}, the leftmost term in \eqref{completeFourierSquare} vanishes. Hence by the equivalence between \eqref{meanSquareDeviationIntegral} and \eqref{completeFourierSquare},

\begin{equation}
\int_D (f - s_N)^2 \cdot \frac{1}{f}d\mu \ = \ \int_D f d\mu - \sum^{N}_{n = 1}\frac{\left(\int_D \phi_n d\mu\right)^2}{\int_D \frac{\phi_n^2}{f}d\mu}.
\label{fourierCoefficientMinimizeDev}
\end{equation}

This is a special case of what is known as \textbf{Bessel's Inequality} \cite{Wei}. However, the left-hand side of \eqref{fourierCoefficientMinimizeDev} is nonnegative, so

\begin{equation}
\int_D f d\mu \ \geq \ \sum^{N}_{n = 1}\frac{\left(\int_D \phi_n d\mu\right)^2}{\int_D \frac{\phi_n^2}{f}d\mu}.
\label{bessel}
\end{equation}

Moreover, since $f > 0$ on $D$, each term in the sum on the lesser side of \eqref{bessel} is nonnegative. Regardless of the value of $N$,

\begin{equation}
\int_D f d\mu \ \geq \ \frac{\left(\int_D \phi_1 d\mu\right)^2}{\int_D \frac{\phi_1^2}{f}d\mu}.
\label{besselTruncated}
\end{equation}

Rearranging the factors in \eqref{besselTruncated} gives us \eqref{integralCSLemma1} immediately. \end{proof}

We now immediately use \eqref{integralCSLemma1} in the proof of another lemma that will complete the reduction of \eqref{integralCS} to the case of integrating functions taking strictly positive values over a bounded, measurable set.

\vspace{0.2 cm}

\begin{lemma}
Let $D \subset \R^n$ be a bounded and measurable set, and let $g, h: D \rightarrow \R \setminus \{0
\}$ be bounded and measurable functions. Then

\begin{equation}
\left(\int_D g^2 d\mu\right)\left(\int_D h^2 d\mu\right) \ \geq \ \left(\int_D gh d\mu\right)^2.
\label{integralCSLemma2}
\end{equation}

\end{lemma}

\begin{proof}
We will perform a change of variables onto $g$ and $h$. Let $f := g^2$ and $\phi_1 := gh$ on $D$. Then the functions $f$, $\frac{\phi_1^2}{f}$, and $\phi_1$ are all bounded and measurable on $D$, and moreover, $f$ is strictly positive on $D$ since $g$ is. With this change of variables the result \eqref{integralCSLemma2} follows immediately from \eqref{integralCSLemma1}. 
\end{proof}

Finally we can consider how to address the case where the functions in question take on the value zero within our choice of measurable set. We hence turn to complete the proof of \eqref{integralCS}.

\begin{proof}[Proof of \eqref{integralCS}] We set $D := \{x \in E. \ g(x) \neq 0 \land h(x) \neq 0\}$, and this set is bounded and measurable since $E$ is. Then the bounded and measurable set $E \setminus D$ is equivalent to $\{x \in E. \ g(x) = 0 \lor h(x) = 0\}$. Since $E = D \cup (E \setminus D)$ is a disjoint union, the identity \eqref{countableIntegralSum} yields

\begin{equation}
\int_E gh d\mu \ = \ \int_D gh d\mu + \int_{E \setminus D}gh d\mu.
\label{lesserSideMeasSplit}
\end{equation}

However, by the choice of set $D$, $gh = 0$ on $E \setminus D$, so the final term in \eqref{lesserSideMeasSplit} vanishes, and \eqref{integralCSLemma2} gives us

\begin{equation}
\left(\int_D g^2 d\mu\right)\left(\int_D h^2 d\mu\right) \ \geq \ \left(\int_E gh d\mu\right)^2.
\label{integralCSLemma2A}
\end{equation}

Now we notice that $g^2$ and $h^2$ are nonnegative on $E \setminus D$, so the following hold:

\begin{equation}
\int_E g^2 d\mu \ = \ \int_D g^2 d\mu + \int_{E \setminus D}g^2 d\mu \ \geq \ \int_D g^2d\mu
\label{nonnegMeasSplitA}
\end{equation}
\begin{equation}
\int_E h^2 d\mu \ = \ \int_D h^2 d\mu + \int_{E \setminus D}h^2 d\mu \ \geq \ \int_D h^2d\mu.
\label{nonnegMeasSplitB}
\end{equation}

Substituting \eqref{nonnegMeasSplitA} and \eqref{nonnegMeasSplitB} into \eqref{integralCSLemma2A} gives us the desired result. \end{proof}

\begin{remark} This line of reasoning would clearly not be valid if we were using Riemann Integrals over intervals in $\R$.
\end{remark}

\vspace{0.2 cm}

\section{Parseval's Identity on Bounded and Measurable Functions}
\label{parsevalGen}

While the Integral Cauchy-Schwarz Inequality is an extremely powerful tool in analysis and partial differential equations, among other fields, the other merit of the proof used in Section \ref{integralCSProof} is that it expedites the development of a special case of Parseval's Identity. Namely, we will investigate the case where the standard Parseval's Identity in \cite{Wei} is reduced to integrating some function $f$ over a measurable set.

\vspace{0.2 cm}

\begin{lemma} [Parseval's Identity on Positive Functions] Let $D \subset \R^n$ be a bounded and measurable set, let $f: D \rightarrow \R$ be bounded, positive, and measurable on $D$, and let $\phi_n: D \rightarrow \R$ be a collection of functions which are mutually orthogonal on $D$ with respect to $\frac{1}{f}$.  Let the constants $c_n$ be defined as in \eqref{fixFourierCoefficients}, and in addition now assume that $f$'s Fourier Expansion actually exists, i.e. $f = \sum^{\infty}_{n = 1}c_n\phi_n$. Then

\begin{equation}
\sum^{\infty}_{n = 1}c_n^2\int_D \frac{\phi_n^2}{f}d\mu \ = \ \int_Df d\mu. \label{posParseval}
\end{equation}

\end{lemma}

\begin{proof} The conditions on $D$ and $f$ for this lemma are the same as for \eqref{integralCSLemma1}. Thus we can re-assert \eqref{bessel}. Since \eqref{bessel} holds $\fa n \in \N^+$, we actually have

\begin{equation}
\int_D fd\mu \ \geq \ \sum^{\infty}_{n = 1}\frac{\left(\int_D \phi_n d\mu\right)^2}{\int_D \frac{\phi_n^2}{f}d\mu}.\label{besselInf}
\end{equation}

Since $f$'s Fourier Expansion exists, we have that $f \ = \ \sum^{\infty}_{n = 1}c_n\phi_n$ and that \eqref{meanSquareDeviationIntegral} will converge to $0$ \cite{Wei}. Then \eqref{posParseval} follows immediately, as \eqref{besselInf} becomes an equality. \end{proof}

\vspace{0.2 cm}

\begin{remark}In the more traditional proof of Parseval's Identity found in \cite{Wei}, the function $f$ we use here is essentially replaced by $\frac{1}{f}$. We made this modification specifically in order to reinterpret Parseval's Identity as a decomposition equalling the integral of a bounded, measurable function over a bounded, measurable set.
\end{remark}

In other words, \eqref{posParseval} establishes a connection between the existence of Fourier Coefficients and the ability to integrate a positive function over a bounded and measurable set. In fact, even if $f$ does not have Fourier Coefficients over a set but does have coefficients on some collection of disjoint, covering subsets, then we can use this theorem multiple times and sum the resulting integrals. It won't even matter if the Fourier Coefficients for $f$ are different on each set. This is demonstrated in the forthcoming proof of \eqref{genParseval}.

\vspace{0.2 cm}

\begin{theorem}[Integration of Bounded and Measurable Functions] Let $E \subset \R^n$ be a bounded and measurable set, let $f: E \rightarrow \R$ be bounded and measurable, and let the sets $D_i$ be bounded, measurable, and mutually disjoint such that $E = \cup^{\infty}_{i = 1}D_i$. Assume that on each $D_i$, $f$ carries a unique sign (i.e. is positive, negative, or zero) and has Fourier Coefficients denoted by 

\begin{equation}
c_{i, n} \ := \ \frac{\int_{D_i} \phi_{i, n} d\mu}{\int_{D_i} \frac{\phi_{i, n}^2}{f}d\mu}
\label{fixFourierCoefficientsIndex}
\end{equation}

for each $i \in \N^+$, where $\phi_{i, n}$ represents a mutually orthogonal family of functions with respect to $\frac{1}{f}$ on $D_i$. Then

\begin{equation}
\sum^{\infty}_{i = 1}\sum^{\infty}_{n = 1}c_{i, n}^2\int_{D_i} \frac{\phi_{i, n}^2}{f}d\mu \ = \ \int_Ef d\mu. \label{genParseval} 
\end{equation}
\end{theorem}
\begin{proof}First we assume without loss of generality that $f$ is nonzero on $E$. If there exists an $i \in \N^+$ for which $f$ is identically zero on $D_i$, then its Fourier Coefficients are all zeros, and so $D_i$ can essentially be ignored when calculating $\int_E f d\mu$. 

\vspace{0.2 cm}

Now let $i \in \N^+$ be arbitrary. If $D_i$ is such that $f > 0$ on $D_i$, then \eqref{posParseval} immediately holds over $D_i$. On the other hand, if $f < 0$ on $D_i$, then $-f > 0$ on $D_i$ so \eqref{posParseval} can be applied to $-f$. We leave the details to the reader to show that this case also reduces to \eqref{posParseval}. That is, $\forall \ i \in \N^+$,

\begin{equation}
\sum^{\infty}_{n = 1}c_{i, n}^2\int_{D_i} \frac{\phi_{i, n}^2}{f}d\mu \ = \ \int_{D_i}f d\mu.\label{parsevalPosIndex}
\end{equation}

Summing \eqref{parsevalPosIndex} over all $i \in \N^+$ and applying \eqref{countableIntegralSum} to the right-hand side gives the desired result. \end{proof}

\vspace{0.2 cm}

\section{Criterion for the Product Inequality}\label{productIneq}

One problem that appears distant from Parseval's Identity and Fourier Analysis is determining for which measurable functions $f$ and $g$ we have the inequality

\begin{equation}\label{prodIneq}
\int_D fg d\mu \geq \int_D f d\mu\int_D g d\mu
\end{equation}

for some bounded measurable set $D$. It turns out we can utilize the results of Section \ref{parsevalGen} to develop a sufficient condition for when this occurs. We state this result explicitly now.

\vspace{0.2 cm}

\begin{theorem} [Sufficient Condition for Product Inequality] Let $D \subset \R^n$ be a bounded and measurable set, let $f, g: D \rightarrow \R$ be bounded, positive, and measurable on $D$, and let $\phi_n, \psi_m: D \rightarrow \R$ be two collections of bounded, measurable functions that are mutually orthogonal on $D$ with respect to $\frac{1}{f}$ and $\frac{1}{g}$, respectively. Suppose that $f$ and $g$ have Fourier Expansions $f = \sum^{\infty}_{n = 1}c_n\phi_n$, $g = \sum^{\infty}_{m = 1}d_m\psi_m$ where 

\begin{equation}\label{firstFourierCoeff}
c_n \ := \ \frac{\int_D \phi_n d\mu}{\int_D \frac{\phi_n^2}{f}d\mu}
\end{equation}
\begin{equation}\label{secondFourierCoeff}
d_m \ := \ \frac{\int_D \psi_md\mu}{\int_D \frac{\psi_m^2}{g}d\mu}
\end{equation}

for all $m, n \in \N^+$. Further suppose that the following inequalities hold for all $m, n \in \N^+$:

\begin{equation}\label{prodFormCritA}
\int_D \frac{\phi_n^2\psi_m^2}{fg}d\mu \ \leq \ \int_D \frac{\phi_n^2}{f}d\mu\int_D\frac{\psi_m^2}{g}d\mu
\end{equation}
\begin{equation}\label{prodFormCritB}
c_nd_m\int_D\phi_n\psi_md\mu \ \geq \ c_nd_m\int_D\phi_nd\mu\int_D\psi_md\mu
\end{equation}

Under these conditions \eqref{prodIneq} holds.
\end{theorem}

\vspace{0.2 cm}

\begin{proof} We can assume without loss of generality that the families $\{\phi_n\}^{\infty}_{n = 1}$ and $\{\psi_m\}^{\infty}_{m = 1}$ both contain infinitely many distinct functions for the same reason as in the proof of \eqref{integralCS}. For each $N \in \N^+$ denote the partial sums $s_N := \sum^{N}_{n = 1}c_n\phi_n$ and $t_N := \sum^{N}_{m = 1}d_m\psi_m$. With that we seek to minimize another mean-square deviation inequality:

\begin{equation}\label{meanSquareDeviationProd}
\int_D(fg - s_Nt_N)^2 \cdot \frac{1}{fg}d\mu.
\end{equation}

Expanding the square yields

\begin{equation}\label{meanSquareDeviationProdA}
\int_D fgd\mu - 2\sum^{N}_{n = 1}\sum^{N}_{m = 1}c_nd_m\int_D \phi_n\psi_md\mu + \int_D \frac{\left(\sum^{N}_{n = 1}c_n\phi_n\sum^{N}_{m = 1}d_m\psi_m\right)^2}{fg}d\mu.
\end{equation}

\vspace{0.2 cm}

Notice that when we expanded the square of \eqref{meanSquareDeviationIntegral}, we used the mutual orthogonality property of $\{\phi_n\}^{\infty}_{n = 1}$ to cause many terms in the expansion of $\int_D \frac{\left(\sum^{N}_{n = 1}c_n\phi_n\right)^2}{fg}d\mu$ to cancel. On the other hand, we do a different analysis here, and henceforth this proof is no longer merely a generalization of the proof of \eqref{integralCS}. We use the Cauchy Schwarz Inequality for finite sums to conclude that \eqref{meanSquareDeviationProdA} is bounded above by

\begin{equation}\label{meanSquareDeviationProdB}
\int_D fgd\mu - 2\sum^{N}_{n = 1}\sum^{N}_{m = 1}c_nd_m\int_D \phi_n\psi_md\mu + \int_D \frac{\sum^{N}_{n = 1}c_n^2\phi_n^2\sum^{N}_{m = 1}d_m^2\psi_m^2}{fg}d\mu.
\end{equation}

Furthermore, rewrite the third integral expression as a double sum of integrals:

\begin{equation}\label{meanSquareDeviationProdC}
\int_D fgd\mu - 2\sum^{N}_{n = 1}\sum^{N}_{m = 1}c_nd_m\int_D \phi_n\psi_md\mu + \sum^{N}_{n = 1}\sum^{N}_{m = 1}c_n^2d_m^2\int_D \frac{\phi_n^2\psi_m^2}{fg}d\mu.
\end{equation}

Now, upon using \eqref{prodFormCritA} and \eqref{prodFormCritB} for each $m, n \leq N$ we see that \eqref{meanSquareDeviationProdC} is bounded above by

\begin{equation}\label{meanSquareDeviationProdD}
\int_D fgd\mu - 2\sum^{N}_{n = 1}\sum^{N}_{m = 1}c_nd_m\int_D \phi_nd\mu\int_D \psi_md\mu + \sum^{N}_{n = 1}\sum^{N}_{m = 1}c_n^2d_m^2\int_D \frac{\phi_n^2}{f}d\mu\int_D\frac{\psi_m^2}{g}d\mu.
\end{equation}

If we add and subtract the term $\sum^{N}_{n = 1}\sum^{N}_{m = 1}\frac{\left(\int_D \phi_nd\mu\right)^2\left(\int_D \psi_md\mu\right)^2}{\int_D \frac{\phi_n^2}{f}d\mu\int_D \frac{\psi_m^2}{g}d\mu}$ from \eqref{meanSquareDeviationProdD} and complete the square we obtain

\begin{equation}\label{meanSquareDeviationProdE}
\sum^{N}_{n = 1}\sum^{M}_{n = 1}\int_D\frac{\phi_n^2}{f}d\mu\int_D\frac{\psi_m^2}{g}d\mu\left(c_nd_m - \frac{\int_D \phi_nd\mu\int_D \psi_md\mu}{\int_D \frac{\phi_n^2}{f}d\mu\int_D\frac{\psi_m^2}{g}d\mu}\right)^2 + \int_D fgd\mu - \sum^{N}_{n = 1}\sum^{N}_{m = 1}\frac{\left(\int_D \phi_nd\mu\right)^2\left(\int_D \psi_md\mu\right)^2}{\int_D \frac{\phi_n^2}{f}d\mu\int_D \frac{\psi_m^2}{g}d\mu}.
\end{equation}

Of course, the squared term vanishes due to the Fourier Coefficient assignments \eqref{firstFourierCoeff} and \eqref{secondFourierCoeff}, and hence it follows that

\begin{equation}\label{meanSquareDeviationProdF}
\int_D (fg - s_Nt_N)^2 \cdot \frac{1}{fg}d\mu \leq \int_D fgd\mu -\sum^{N}_{n = 1}\sum^{N}_{m = 1}\frac{\left(\int_D \phi_nd\mu\right)^2\left(\int_D \psi_md\mu\right)^2}{\int_D \frac{\phi_n^2}{f}d\mu\int_D \frac{\psi_m^2}{g}d\mu}.
\end{equation}

The lesser side of \eqref{meanSquareDeviationProdF} is nonnegative, so bounding it below by $0$ and rearranging yields another Bessel-type inequality

\begin{equation}\label{meanSquareDeviationProdG}
\int_D fgd\mu \geq \sum^{N}_{n = 1}\sum^{N}_{m = 1}\frac{\left(\int_D \phi_nd\mu\right)^2\left(\int_D \psi_md\mu\right)^2}{\int_D \frac{\phi_n^2}{f}d\mu\int_D \frac{\psi_m^2}{g}d\mu}.
\end{equation}

We can readily write the lesser side of \eqref{meanSquareDeviationProdG} as a product of two sums:

\begin{equation}\label{meanSquareDeviationProdH}
\int_D fgd\mu \geq \left(\sum^{N}_{n = 1}\frac{\left(\int_D \phi_nd\mu\right)^2}{\int_D \frac{\phi_n^2}{f}d\mu}\right)\left(\sum^{N}_{m = 1}\frac{\left(\int_D \psi_md\mu\right)^2}{\int_D \frac{\psi_m^2}{g}d\mu}\right).
\end{equation}

Since the greater side of \eqref{meanSquareDeviationProdH} is independent of $N$ we can take the limit $N \rightarrow \infty$ to realize

\begin{equation}\label{meanSquareDeviationProdI}
\int_D fgd\mu \geq \left(\sum^{\infty}_{n = 1}\frac{\left(\int_D \phi_nd\mu\right)^2}{\int_D \frac{\phi_n^2}{f}d\mu}\right)\left(\sum^{\infty}_{m = 1}\frac{\left(\int_D \psi_md\mu\right)^2}{\int_D \frac{\psi_m^2}{g}d\mu}\right).
\end{equation}

In particular the lesser side is finite because the integral $\int_D fgd\mu$ is finite. Since the families $\{\phi_n\}^{\infty}_{n = 1}, \{\psi_m\}^{\infty}_{m = 1}$ are mutually orthogonal on $D$ with respect to $\frac{1}{f}$ and $\frac{1}{g}$ respectively, and we have the Fourier expansions $f = \sum^{\infty}_{n = 1}c_n\phi_n$, $g = \sum^{\infty}_{m = 1}d_m\psi_m$, we can use the Parseval Identity \eqref{posParseval} twice to immediately conclude \eqref{prodIneq}. \end{proof}

\vspace{0.2 cm}

\begin{remark} Notice that we have

\begin{equation}\label{prodFormCritACSInt}
\int_D \frac{\phi_n^2\psi_m^2}{fg}d\mu \leq \sqrt{\int_D\frac{\phi_n^4}{f^2}d\mu\int_D\frac{\psi_m^4}{g^2}d\mu}
\end{equation}

as a corollary of \eqref{integralCS}. In this case, \eqref{prodFormCritA} follows if $\int_D \frac{\phi_n^4}{f^2}d\mu \leq \left(\int_D \frac{\phi_n^2}{f}d\mu\right)^2$ for all $n \in \N^+$ and $\int_D \frac{\psi_m^4}{g^2}d\mu \leq \left(\int_D \frac{\psi_m^2}{g}d\mu\right)^2$. This gives us a sense of what type of behavior is required for the individual functions $\{\phi_n\}^{\infty}_{n = 1}$ and $\{\psi_m\}^{\infty}_{m = 1}$. \end{remark}

\vspace{0.2 cm}

We notice that the proof we just completed not only utilizes \eqref{posParseval} twice, but it also serves to magnify the significance of the proofs given in Section \ref{integralCSProof} because the proofs rest on the same fundamental idea and setup. Moreover, observe that in this proof the mutual orthogonality of the families $\{\phi_n\}^{\infty}_{n = 1}$ and $\{\psi_m\}^{\infty}_{m = 1}$ was not invoked explicitly, but only implicitly in justifying the use of \eqref{posParseval}. 

\vspace{0.2 cm}
\section{Future Work}\label{future}

The potential extensions of the theory used to reinvent Parseval's Identity are abundant. In particular, one natural extension is to revert the Lebesgue integrals in \eqref{genParseval} to Riemann Integrals and use the formula to perform explicit calculations. The book \cite{Wei} provides some explicit choices for mutually orthogonal families of functions to serve this purpose. 

\vspace{0.2 cm}

There also exist more theoretical extensions of the work done that penetrate into other fields of mathematics. For instance, if one writes the mean-square deviation integral \eqref{meanSquareDeviationIntegral} in the form of a generic energy functional

\begin{equation}
\int_D G(f, \{\phi_n\}, N)d\mu, \label{meanSquareCoV}
\end{equation}

then utilization of techniques from calculus of variations can be used to address problems in the existence of minimizers. That is, for a choice of $f$ and $N$, which choice of mutually orthogonal functions $\{\phi_n\}^{N}_{n = 1}$ minimizes the value of \eqref{meanSquareCoV}, if such a family even exists. In a similar vein one may ask how the decay rate of \eqref{meanSquareDeviationIntegral} changes with respect to choice of function $f$, and alternatively if the square exponent is replaced with a higher power. 

\vspace{0.2 cm}

While inequality \eqref{prodIneq} was introduced primarily as an application of the material in Sections \ref{integralCSProof} and \ref{parsevalGen}, it is interesting in its own right. One may wish to explore necessary conditions for the inequality to hold, or alternatively the sharpness of the inequality.

\vspace{0.2 cm}

Finally, there is a problem of potential interest to numerical analysts. The measure-theoretic techniques of changing the variables and gradual construction of measurable sets used in Section \ref{integralCSProof} are likely reusable to find generalizations of the Trapezoid Rule Inequalities found in \cite{Dra, Nwa, Pac}. In particular we ask if similar bounds exist when the Riemann Integrals are replaced with Lebesgue Integrals over bounded measurable subsets of $\R$.

\vspace{0.2 cm}

\section{Acknowledgments}

The author wishes to thank Shlomo Ta'asan for providing feedback on the content and delivery of proofs presented in this note. He also wishes to thank Akanksha Kartik for proofreading early drafts of the manuscript. 


\end{document}